\documentclass[reqno,centertags,12pt]{amsart}

\usepackage{amssymb,amsmath,amsthm,dsfont}

\makeatletter \def\@strippedMR{} \def\@scanforMR#1#2#3\endscan{%
  \ifx#1M\ifx#2R\def\@strippedMR{#3}%
  \else\def\@strippedMR{#1#2#3}%
  \fi\fi} \renewcommand\MR[1]{\relax \ifhmode\unskip\spacefactor3000
  \space\fi \@scanforMR#1\endscan
  MR\MRhref{\@strippedMR}{\@strippedMR}} \makeatother

\newcommand{\R}{\mathbb{R}} \newcommand{\Z}{\mathbb{Z}}
\newcommand{\T}{\mathbb{T}} \newcommand{\C}{\mathbb{C}}
\renewcommand{\S}{\mathbb{S}}

\theoremstyle{plain} \newtheorem{theorem}{Theorem}[section]
\newtheorem{lemma}[theorem]{Lemma}
\newtheorem{coro}[theorem]{Corollary}
\newtheorem{prop}[theorem]{Proposition}

\theoremstyle{definition} \newtheorem{definition}[theorem]{Definition}
\theoremstyle{remark} \newtheorem{remark}{Remark}

\DeclareMathOperator{\supp}{supp}

\newcommand{\eps}{\varepsilon} \newcommand{\lb}{\langle}
\newcommand{\rb}{\rangle}
\newcommand{\ls}{\lesssim}

\begin{document}

\title[GWP of the energy critical NLS with small data in $H^1(\T^3)$]%
{Global well-posedness of the energy critical Nonlinear Schr\"odinger
  equation with small initial data in $H^1(\T^3)$}
\author[S.~Herr]{Sebastian~Herr}\author[D.~Tataru]{Daniel~Tataru}
\author[N.~Tzvetkov]{Nikolay~Tzvetkov}

\subjclass[2000]{35Q55}

\address{Mathematisches Institut, Universit\"at Bonn, Endenicher Allee
  60, 53115 Bonn, Germany} \email{herr@math.uni-bonn.de}
\address{Department of Mathematics, University of California,
  Berkeley, CA 94720-3840, USA} \email{tataru@math.berkeley.edu}
\address{D\'epartement de Math\'ematiques, Universit\'e de
  Cergy-Pontoise, 2, avenue Adolphe Chauvin, 95302 Cergy-Pontoise
  Cedex, France \and Institut Universitaire de France} \email{nikolay.tzvetkov@u-cergy.fr}

\begin{abstract}
  A refined trilinear Strichartz estimate for solutions to the
  Schr\"odinger equation on the flat rational torus $\T^3$ is derived.
  By a suitable modification of critical function space theory this is
  applied to prove a small data global well-posedness result for the
  quintic Nonlinear Schr\"odinger Equation in $H^s(\T^3)$ for all
  $s\geq 1$. This is the first energy-critical global well-posedness
  result in the setting of compact manifolds.
\end{abstract}
\keywords{}
\maketitle
\section{Introduction and main result}\label{sect:intro_main}
\noindent
Our goal here is to establish a critical local well-posedness theory
for the nonlinear Schr\"odinger equation
\begin{equation}\label{eq:nls_0}
  (i\partial_t+\Delta) u=|u|^4u,
\end{equation}
posed on the three dimensional torus $\T^3=\R^3/(2\pi \Z)^3$ and thus
to naturally extend the result of Bourgain's fundamental paper
\cite{B93a} where only sub-critical regularity is considered.  To our
knowledge, this is the first critical well-posedness result in the
case of a nonlinear Schr\"odinger equation on a compact manifold.

Since the critical space associated to \eqref{eq:nls_0} is the Sobolev
space $H^1$, our result will directly imply small data global
well-posedness for \eqref{eq:nls_0} (see the conservation of energy
\eqref{eq:energy-cons} below).  The problem of arbitrary data global
well-posedness is a remaining challenging issue.

If the problem \eqref{eq:nls_0} is posed on the Euclidean space $\R^3$
then one obtains small data global well-posedness by invoking the
Strichartz estimate (see e.g. \cite{KT98})
\begin{equation}\label{strichartz}
  \|u\|_{L^\infty_{t}H^s_{x}}+\|u\|_{L^2_t;W^{s,6}_x}\lesssim
  \|u(0)\|_{H^s}+\|(i\partial_t +\Delta) u\|_{L^2_t;W^{s,6/5}_x}\,.
\end{equation}
Let us give the argument (see \cite{CW90}).  Applying
\eqref{strichartz} with $s=1$ in the context of \eqref{eq:nls_0}
together with the H\"older and the Sobolev embedding
$H^1\hookrightarrow L^6$ yield
\begin{equation*}
  \|u\|_{L^\infty_{t}H^1_{x}}+\|u\|_{L^2_t;W^{1,6}_x}
  \lesssim 
  \|u(0)\|_{H^1}+\|u\|_{L^2_t W^{1,6}_x}\|u\|_{L^\infty_t H^1_x}^4\,.
\end{equation*}
The above a priori estimate easily transforms to small data global
well-posedness result for \eqref{eq:nls_0} by invoking the Picard
iteration scheme (the use of the endpoint Strichartz estimate in this
reasoning is not really needed).  Recently, even the global
well-posedness and scattering in the Euclidean setting has been
established for large data by
Colliander--Keel--Staffilani--Takaoka--Tao \cite{CKSTT08}.

The arguments of the Euclidean setting completely fail if
\eqref{eq:nls_0} is posed on a compact Riemannian manifold. The
Strichartz inequality \eqref{strichartz} does not hold in the periodic
case as one can see it by adapting to the 3d situation the 1d
counterexample of \cite{B93a}.  It also strongly fails, if $\Delta$ is
the Laplace-Beltrami operator on the three dimensional sphere by
testing it for instance on zonal eigenfunctions with large eigenvalues
(see \cite{BGT04}).  The most dramatic failure is that of the
non-homogeneous estimate.

As shown by the work of Bourgain \cite{B93a}, Burq-Gerard-Tzvetkov
\cite{BGT04, BGT05b,BGT05a,BGT07} some weak versions of the inequality
\eqref{strichartz} survive in the setting of compact manifolds.  For
instance one has that for free solutions to the Schr\"odinger equation
on $\T^3$, we have the estimate
\begin{equation}\label{loss}
  \|u\|_{L^4_t L^4_x}\lesssim\|u(0)\|_{H^{\frac{1}{4}+}}\,.
\end{equation}
In the Euclidean setting the same estimate without any loss is
obtained from \eqref{strichartz} combined with the Sobolev embedding
$W^{\frac{1}{4},3}(\R^3)\hookrightarrow L^4(\R^3)$ (by interpolation
one may also put $L^4_tL^3_x$ type norms in \eqref{strichartz}).  A
similar estimate holds if the torus is replaced with the three
dimensional sphere; there the epsilon loss can not be avoided, see
\cite{BGT04}.

Using the estimate \eqref{loss} and its bilinear extensions one can
prove essentially optimal (up to the critical regularity)
well-posedness results for the nonlinear Schr\"odinger equation, on
both the $3d$ rational torus or the $3d$ sphere (see \cite{B93a},
\cite{BGT05a}). We also refer to \cite{B07} for partial results in the
case of irrational tori.

In the present paper we shall be able to prove the endpoint result,
i.e.  well-posedness in the energy space $H^1$.  For this purpose
estimates with an epsilon loss of type \eqref{loss} are useless. Our
strategy is to use trilinear Strichartz type estimates on the right
scales (with no loss with respect to scaling) only, involving $L^p_t$
with $p>2$, in the context of the $U^p$ and $V^p$ type critical spaces
theory.  These function spaces were originally developed in
unpublished work on wave maps by the second author, see also
Koch-Tataru \cite{KoTa05}.  Here we will use the formalism as
presented in Hadac-Herr-Koch \cite{HHK09}, but with the following
refinement. The scale-invariant linear Strichartz estimates which we
aim to trilinearize involve frequency scales which are finer than the
standard dyadic localizations. Therefore, a modification of the norms
is necessary which allows us to distinguish finer scales.

Let us now describe more precisely our results. Consider the Cauchy
problem
\begin{equation}\label{eq:nls}
  \begin{split}
    i\partial_t u +\Delta u&= |u|^4u\\
    u(0,x)&=\phi(x) \in H^s(\T^3).
  \end{split}
\end{equation}
For strong solutions $u$ of \eqref{eq:nls} we have energy conservation
\begin{equation}\label{eq:energy-cons}
  E(u(t))=\frac12 \int_{\T^3}|\nabla u(t,x)|^2dx+ \frac16
  \int_{\T^3}|u(t,x)|^6 dx=E(\phi),
\end{equation}
and $L^2$-conservation
\begin{equation}\label{eq:l2-cons}
  M(u(t))=\frac12 \int_{\T^3}|u(t,x)|^2dx=M(\phi).
\end{equation}
Thus, the natural energy space for this equation is $H^1(\T^3)$.  The
same problem \eqref{eq:nls} but considered in $\R^3$ is invariant with
respect to the scaling law
\[
u(t,x) \to u_\lambda(t,x)=\lambda^{1/2} u(\lambda^2 t, \lambda x)
\]
The energy is also invariant with respect to this scaling, which is
why this problem is called energy critical. Scaling considerations
remain valid in the periodic case for space-time scales which are $\ll
1$.

Define
\[
B_{\eps}(\phi_\ast):=\{\phi \in H^1(\T^3):
\|\phi-\phi_{\ast}\|_{H^1}<\eps\}.
\]
Let us now state our first result dealing with local well-posedness.
\begin{theorem}[Local well-posedness]\label{thm:main-local}
  Let $s\geq 1$. For every $\phi_{\ast}\in H^1(\T^3)$ there exists
  $\eps>0$ and $T=T(\phi_{\ast})>0$ such that for all initial data
  $\phi\in B_{\eps}(\phi_\ast)\cap H^s(\T^3)$ the Cauchy problem
  \eqref{eq:nls} has a unique solution \[u \in C([0,T);H^s(\T^3)) \cap
  X^s([0,T)).\] This solution obeys the conservation laws
  \eqref{eq:energy-cons} and \eqref{eq:l2-cons}, and the flow map
  \[
  B_{\eps}(\phi_\ast)\cap H^s(\T^3)\ni \phi\mapsto u \in
  C([0,T);H^s(\T^3)) \cap X^s([0,T))
  \]
  is Lipschitz continuous.
\end{theorem}

The spaces $X^s$ involved in this statement are the critical spaces
associated to our problem (see \eqref{eq:xtnorm} below). In the case
of small data, the proof of Theorem \ref{thm:main-local} shows that
the life span $T$ of the solution only depends on the $H^1$-norm of
the initial data. Thanks to the conservation laws, we therefore have a
global in time result:

\begin{theorem}[Small data global well-posedness]\label{thm:main-global}
  Let $s\geq 1$. There exists $\eps_0>0$ such that for all initial
  data $\phi\in B_{\eps_0}(0)\cap H^s(\T^3)$ and every $T>0$ the
  Cauchy problem \eqref{eq:nls} has a unique solution \[u \in
  C([0,T);H^s(\T^3))\cap X^s([0,T)).\] This solution obeys the
  conservation laws \eqref{eq:energy-cons} and \eqref{eq:l2-cons}, and
  the flow map
  \[
  B_{\eps_0}(0)\cap H^s(\T^3)\ni \phi\mapsto u \in
  C([0,T);H^s(\T^3))\cap X^s([0,T))
  \]
  is Lipschitz continuous.
\end{theorem}

The results of Theorem~\ref{thm:main-local} and
Theorem~\ref{thm:main-global} hold equally well for the focusing
problem
\[
i\partial_t u +\Delta u= -|u|^4u\,.
\]
On the other hand in the large data global theory the difference
between these two cases should of course be crucial.

At the moment we are not able to prove well-posedness in the energy
space for \eqref{eq:nls_0} on the three dimensional sphere $\S^3$.
Some progress in this direction was achieved in \cite{BGT07} where one
gains control on the second Picard iteration in $H^1$. However the
argument in \cite{BGT07} uses strongly separated interactions of the
spatial and the time frequencies, a fact which does not fit well with
the critical spaces machinery. In particular the control of the second
Picard iteration performed in \cite{BGT07} does not rely only on
$L^p_t$, $p>2$ properties of free solutions as is the case in the
present paper.

The outline of the paper is as follows: In Section \ref{sect:crit_fs}
we introduce some notation and the critical function spaces. In Section
\ref{sect:str} we prove trilinear Strichartz type estimates which
extend the results of J. Bourgain \cite{B93a}. In Section \ref{sect:pr} we show how
these estimates imply our main results stated above.

  We remark that for a proof of the crucial multilinear estimate in
  Proposition \ref{prop:multilinear_est} the
  refinements involving localizations to rectangles in Section \ref{sect:str} are not
  necessary.
  Indeed, estimate \eqref{eq:multilinear_est} could be
  derived from the dyadic tri-linear Strichartz type estimate
  \begin{equation}\label{eq:l2_est_alt}
    \|\prod_{j=1}^3P_{N_j} u_j\|_{L^2}\ls
    N_2^{1-\delta}N_3^{1+\delta}
    \prod_{j=1}^3\|P_{N_j}u_j\|_{Y^0},
  \end{equation}
  for some small $\delta>0$, which can be proved directly by
  orthogonality considerations, H\"older's inequality and the extension of Bourgain's estimate to critical function spaces
  \eqref{eq:str_emb_cubes}. However, if one follows this strategy of
  proof, critical spaces which are sensible to finer than dyadic
  scales would still be needed.

  We decided to present the stronger estimates
  because we believe that they are interesting in their own right,
  also for linear solutions. They may also prove useful in the study of the large data case.
  The technique we will use involves
  additional orthogonality considerations with respect to the temporal
  frequency of the solutions, not only with respect to the spatial
  frequencies. This circle of ideas has further applications, e.g. to
  $4d$ cubic Schr\"odinger equations, see \cite{HTT10b}.

\subsection*{Acknowledgments} The authors are grateful to the anonymous referees for their valuable comments. In particular, a suggestion of one of the referees lead to a modification and simplification of Proposition \ref{prop:str_linear_impr_linfty}, and as a consequence to an improvement of Corollary \ref{coro:str_linear_impr}.

\section{Critical function spaces}\label{sect:crit_fs}
\noindent
Throughout this section let $H$ be a separable Hilbert space over
$\C$. In our application this will be chosen as either $H^s(\T^3)$ or
$\C$. We will briefly introduce the function spaces $U^p$ and $V^p$
(we refer the reader to \cite[Section 2]{HHK09} for detailed proofs of
the basic properties) and use them to construct all relevant function
spaces for the present paper.

Let $\mathcal{Z}$ be the set of finite partitions
$-\infty<t_0<t_1<\ldots<t_K\leq \infty$ of the real line. If
$t_K=\infty$, we use the convention that $v(t_K):=0$ for all functions
$v: \R \to H$. Let $\chi_I:\R\to\R$ denote the sharp characteristic
function of a set $I\subset \R$.

\begin{definition}\label{def:u}
  Let $1\leq p <\infty$. For $\{t_k\}_{k=0}^K \in \mathcal{Z}$ and
  $\{\phi_k\}_{k=0}^{K-1} \subset H$ with
  $\sum_{k=0}^{K-1}\|\phi_k\|_{H}^p=1$ we call the piecewise
  defined function $a:\R \to H$,
  \begin{equation*}
    a=\sum_{k=1}^K\chi_{[t_{k-1},t_k)}\phi_{k-1}
  \end{equation*}
  a $U^p$-atom and we define the atomic space $U^p(\R,H)$ of all
  functions $u: \R \to H$ such that
  \begin{equation*}
    u=\sum_{j=1}^\infty \lambda_j a_j \; \text{ for } U^p\text{-atoms } a_j,\;
    \{\lambda_j\}\in \ell^1,
  \end{equation*}
  with norm
  \begin{equation}\label{eq:norm_u}
    \|u\|_{U^p}:=\inf \left\{\sum_{j=1}^\infty |\lambda_j|
      :\; u=\sum_{j=1}^\infty \lambda_j a_j,
      \,\lambda_j\in \C,\; a_j \text{ $U^p$-atom}\right\}.
  \end{equation}
\end{definition}

\begin{remark}\label{rmk:prop_up} The spaces $U^p(\R,H)$ are Banach spaces and
  we observe that $U^p(\R,H)\hookrightarrow L^\infty(\R;H)$. Every
  $u\in U^p(\R,H)$ is right-continuous and $u$ tends to $0$ for $t \to
  -\infty$.
\end{remark}

\begin{definition}\label{def:v}
  Let $1\leq p<\infty$.
  \begin{enumerate}
  \item We define $V^p(\R,H)$ as the space of all functions $v:\R\to
    H$ such that
    \begin{equation}\label{eq:norm_v}
      \|v\|_{V^p}
      :=\sup_{\{t_k\}_{k=0}^K \in \mathcal{Z}} \left(\sum_{k=1}^{K}
        \|v(t_{k})-v(t_{k-1})\|_{H}^p\right)^{\frac{1}{p}}
    \end{equation}
    is finite\footnote{Notice that here we use the convention
      $v(\infty)=0$}.
  \item Likewise, let $V^p_{rc}(\R,H)$ denote the closed subspace of
    all right-continuous functions $v:\R\to H$ such that $\lim_{t\to
      -\infty}v(t)=0$, endowed with the same norm \eqref{eq:norm_v}.
  \end{enumerate}
\end{definition}

\begin{remark}\label{rmk:prop_v}
  The spaces $V^p(\R,H)$, $V^p_{rc}(\R,H)$ are Banach spaces and
  satisfy $U^p(\R,H)\hookrightarrow V_{rc}^p(\R,H)\hookrightarrow
  L^\infty(\R;H)$.
\end{remark}

\begin{prop}\label{prop:emb_v_in_u} For  $1\leq p<q <\infty$
  we have $V^p_{rc}(\R,H) \hookrightarrow U^q(\R,H)$.
\end{prop}

We also record a useful interpolation property of the spaces $U^p$ and
$V^p$, cf. \cite[Proposition 2.20]{HHK09}.
\begin{lemma}\label{lem:interpol}
  Let $q_1,q_2,q_3>2$, $E$ be a Banach space and
  \[T:U^{q_1}\times U^{q_2}\times U^{q_3}\to E\] be a bounded,
  tri-linear operator with $\|T(u_1,u_2,u_3)\|_E \leq C
  \prod_{j=1}^3\|u_j\|_{U^{q_j}}$.  In addition, assume that there
  exists $C_2\in (0,C]$ such that the estimate $\|T(u_1,u_2,u_3)\|_E
  \leq C_2 \prod_{j=1}^3\|u_j\|_{U^2}$ holds true.  Then, $T$
  satisfies the estimate
  \[
  \|T(u_1,u_2,u_3)\|_E \ls C_2
  (\ln\frac{C}{C_2}+1)^3\prod_{j=1}^3\|u_j\|_{V^2}, \quad u_j \in
  V^2_{rc},\; j=1,2,3.
  \]
\end{lemma}
\begin{proof}
  For fixed $u_2,u_3$ let $T_1u:=T(u,u_2,u_3)$. From the assumption we
  have
  \begin{align*}
    \|T_1u\|_E&\leq D_1 \|u\|_{U^{q_1}} \text{ where }
    D_1=C \|u_2\|_{U^{q_2}}\|u_3\|_{U^{q_3}},\\
    \|T_1u\|_E&\leq D_1' \|u\|_{U^{2}} \text{ where } D_1'=C_2
    \|u_2\|_{U^{2}}\|u_3\|_{U^{2}}.
  \end{align*}
  An application of \cite[Proposition 2.20]{HHK09} and the bound
  $\|u_j\|_{U^{q_j}}\leq \|u_j\|_{U^{2}}$ for $j=2,3$ yield
  \[
  \|T_1u\|_E \ls C_2 (\ln \frac{C}{C_2} +1)
  \|u\|_{V^2}\|u_2\|_{U^{2}}\|u_3\|_{U^{2}}.
  \]
  The embedding $V^2_{rc}\hookrightarrow U^{q_j}$ allows us to repeat
  this argument with respect to the second and third argument of $T$.
\end{proof}

We define the spatial Fourier coefficients
\[
\widehat{f}(\xi):=(2\pi)^{-3/2}\int_{ [0,2\pi]^{3}} e^{-ix\cdot \xi}
f(x) \;dx,\; \xi \in \Z^{3}.
\]
and the space time Fourier transform
\[
\mathcal{F}u(\tau,\xi):=(2\pi)^{-2}\int_{ \R\times [0,2\pi]^{3}}
e^{-i(x\cdot \xi+t\tau)} u(t,x) \;dt dx,\; (\tau,\xi) \in
\R\times\Z^{3}.
\]

We fix a non-negative, even function $\psi\in C_0^\infty((-2,2))$ with
$\psi(s)=1$ for $|s|\leq 1$ in order to define a partition of unity:
for a dyadic number $N\geq 1$ we set
\[
\psi_N(\xi)=\psi\Big(\frac{|\xi|}{N}\Big)-\psi\Big(\frac{2|\xi|}{N}\Big),
\quad \text{for } N\geq 2, \quad \psi_1(\xi)=\psi(|\xi|).
\]
We define the frequency localization operators $P_N:L^2(\T^3)\to
L^2(\T^3)$ as the Fourier multiplier with symbol $\psi_N$, and for
brevity we also write $u_N:=P_Nu$. Moreover, we define $P_{\leq
  N}:=\sum_{1\leq M\leq N}P_M$.

More generally, for a set $S\subset \Z^{3}$ we define $P_S$ as the Fourier
projection operator with symbol $\chi_S$, where $\chi_S$ denotes the
characteristic function of $S$.

Let $s \in \R$.  We define the Sobolev space $H^s(\T^3)$ as the space
of all $L^2(\T^3)$-functions for which the norm
\[
\|f\|_{H^s(\T^3)}:=\left(\sum_{\xi \in \Z^3} \lb \xi\rb^{2s} |\hat
  f(\xi)|^2\right)^{\frac12} \approx \left(\sum_{N \geq 1}N^{2s} \|P_N
  f\|_{L^2(\T^3)}^2\right)^{\frac12}
\]
is finite.

Corresponding to the linear Schr\"odinger flow we define
\begin{definition}\label{def:delta_norm}
  For $s \in \R$ we let $U^p_\Delta H^s $ resp. $V^p_\Delta H^s$ be
  the spaces of all functions $u:\R\to H^s(\T^3)$ such that $t \mapsto
  e^{-it \Delta}u(t)$ is in $U^p(\R,H^s)$ resp. $V^p(\R,H^s)$, with
  norms
  \begin{equation}\label{eq:delta_norm}
    \| u\|_{U^p_\Delta H^s} = \| e^{-it \Delta} u\|_{U^p(\R,H^s)},
    \qquad 
    \| u\|_{V^p_\Delta H^s} = \| e^{-it \Delta} u\|_{V^p(\R,H^s)}.
  \end{equation}
\end{definition}
Spaces of this type have been succesfully used as replacements for
$X^{s,b}$ spaces which are still effective at critical scaling, see
for instance \cite{KoTa05}, \cite{KoTa07}, \cite{HHK09}.

\begin{remark}\label{rem:ext}
  Proposition \ref{prop:emb_v_in_u} and Lemma \ref{lem:interpol}
  naturally extend to the spaces $U^p_{\Delta}H^s$ and
  $V^p_{\Delta}H^s$.
\end{remark}

In the context of this article this would suggest that spaces of the
form $U^p_\Delta H^1$ and $V^p_\Delta H^1$ should be useful.  However,
here we introduce a small variation on this theme.

\begin{definition}\label{def:xt}
  For $s \in \R$ we define $X^{s}$ as the space of all functions $u:
  \R \to H^s(\T^3)$ such that for every $\xi \in \Z^3$
  the map $t \mapsto e^{it|\xi|^2} \widehat{u(t)}(\xi)$ is in
  $U^2(\R,\C)$, and for which the norm
  \begin{equation}\label{eq:xtnorm}
    \|u\|_{X^{s}}:=\left(\sum_{\xi \in \Z^3}\lb \xi\rb ^{2s}
      \|e^{it|\xi|^2}\widehat{u(t)}(\xi)\|_{U^2_t}^2\right)^{\frac12}
  \end{equation}
  is finite.
\end{definition}

\begin{definition}\label{def:yt}
  Let $s \in \R$. We define $Y^{s}$ as the space of all functions
  $u:\R \to H^s(\T^3)$ such that for every $\xi\in \Z^3$
  the map $t \mapsto e^{it|\xi|^2}\widehat{u(t)}(\xi)$ is in
  $V^2_{rc}(\R,\C)$, and for which the norm
  \begin{equation}\label{eq:ytnorm}
    \|u\|_{Y^{s}}
    :=\left(\sum_{\xi\in \Z^3}
      \lb\xi\rb ^{2s} \|e^{it|\xi|^2}\widehat{u(t)}(\xi)\|_{V^2_t}^2\right)^{\frac12}
  \end{equation}
  is finite.
\end{definition}

As usual, for a time interval $I\subset \R$ we also consider the
restriction spaces $X^s(I)$, etc. It is easy to relate the $X^s$ and
$Y^s$ spaces with the previously defined $V^p_\Delta H^s$ and
$U^p_\Delta H^s $. Using also Remark \ref{rmk:prop_v}, we have:

\begin{prop}\label{prop:x_in_y}
  The following continuous embeddings hold:
  \[
  U^2_\Delta H^s \hookrightarrow X^s \hookrightarrow Y^s
  \hookrightarrow V^2_\Delta H^s.
  \]
\end{prop}
The motivation for the introduction of the $X^s$ and $Y^s$ spaces lies
in the following

\begin{coro}
  Let $\Z^3 = \cup C_k$ be a partition of $\Z^3$. Then
  \begin{equation}
    \left(\sum_k \| P_{C_k} u\|_{V^2_\Delta H^s}^2\right)^{\frac12} \lesssim  \| u\|_{Y^s}.
  \end{equation}
\end{coro}

The following Proposition follows immediately from the atomic
structure of $U^2$.
\begin{prop}\label{prop:linear}
  Let $s\geq 0$, $0<T\leq \infty$ and $\phi \in H^s(\T^3)$. Then, for
  the linear solution $u(t):=e^{it\Delta}\phi$ for $t\geq 0$ we have
  $u \in X^s([0,T))$ and
  \begin{equation}\label{eq:linear}
    \|u\|_{X^s([0,T))}\leq \|\phi\|_{H^s}.
  \end{equation}
\end{prop}

Let $f\in L^1_{loc}([0,\infty);L^2(\T^3))$ and define
\begin{equation}\label{eq:duhamel}
  \mathcal{I}(f)(t):=\int_{0}^t e^{i(t-s)\Delta} f(s) ds
\end{equation}
for $t \geq 0$ and $\mathcal{I}(f)(t)=0$ otherwise. We have the
following linear estimate for the Duhamel term.
\begin{prop}\label{prop:inhom_est}
  Let $s \geq 0$ and $T>0$. For $f \in L^1([0,T);H^s(\T^3))$ we have
  $\mathcal{I}(f)\in X^s([0,T))$ and
  \begin{equation}\label{eq:inhom_est}
    \|\mathcal{I}(f)\|_{X^s([0,T))}
    \leq \sup_{v \in Y^{-s}([0,T)): \|v\|_{Y^{-s}}=1}
    \left|\int_0^T \int_{\T^3}f(t,x)\overline{v(t,x)} dx
    dt\right| 
  \end{equation}
\end{prop}
\begin{proof}
  We extend $\mathcal{I}(f)$ continuously by
  $\mathcal{I}(f)(t):=\int_{0}^T e^{i(t-s)\Delta} f(s) ds$ for $t>T$.
  For each $\xi$ we have that $t \to
  e^{it|\xi|^2}\widehat{\mathcal{I}(f)(t)}(\xi)$ is absolutely
  continuous and of bounded variation, hence in $U^2$. Let $\eps>0$.
  There is a sequence $(b_\xi)_{\xi \in \Z^3}\in \ell^2(\Z^3)$,
  $\|(b_\xi)\|_{\ell^2}=1$ such that
  \[
  \|\mathcal{I}(f)\|_{X^s([0,T))}\leq \sum_{\xi \in \Z^3} b_\xi \lb
  \xi\rb^s
  \left\|\int_0^t\chi_{[0,T)}(s)e^{is|\xi|^2}\widehat{f(s)}(\xi)ds\right\|_{U^2_t}+\eps.
  \]
  By duality of $U^2$ and $V^2$, see \cite [Theorem 2.8 and
  Proposition 2.10]{HHK09}, for each $\xi\in \Z^3$ there exists
  $v_\xi\in V^2_t$, $\|v_\xi\|_{V^2}=1$, such that
  \[
  \left\|\int_0^t\chi_{[0,T)}(s)e^{is|\xi|^2}\widehat{f(s)}(\xi)ds\right\|_{U^2_t}\leq\left| 
  \int_0^T \widehat{f(s)}(\xi)\overline{e^{-is|\xi|^2} v_\xi(s)}
  ds\right| +2^{-|\xi|^2}\eps.
  \]
  Evidently, we can choose $v_\xi$ to be right-continuous with $\supp
  v_\xi\subseteq [0,T)$ without changing the above expression.  Define
  \[
  v(t,x)=(2\pi)^{-3/2}\sum_{\xi\in \Z^3} b_\xi \lb \xi\rb^s e^{ix\cdot
    \xi}e^{-it|\xi|^2} v_\xi(t).
  \]
  Then, $v \in Y^{-s}([0,T))$, $\|v\|_{Y^{-s}}\leq 1$, and
  \[
  \|\mathcal{I}(f)\|_{X^s([0,T))}\leq \sum_{\xi \in \Z^3} \left| \int_0^T
  \widehat{f(t)}(\xi) \overline{\widehat{v(t)(\xi)}} dt\right|  +c\eps,
  \]
  and the claim follows by dominated convergence and Plancherel's
  identity.
\end{proof}

\begin{remark}\label{rem:l1bound}
In particular, Proposition \ref{prop:inhom_est} implies that any
estimate on $f$ in $L^1([0,T);H^s(\T^3))$ bounds $\mathcal{I}(f)$ in
$X^s([0,T))$.
\end{remark}
\section{Strichartz estimates}\label{sect:str}
We first recall Bourgain's $L^p$ estimates of Strichartz type.  Let
$\mathcal{C}_N$ denote the collection of cubes $C\subset \Z^3$ of
side-length $N\geq 1$ with arbitrary center and orientation.

\begin{prop}[Bourgain \cite{B93a}]\label{prop:str}
  Let $p>4$. For all $N\geq 1$ we have
  \begin{equation}\label{eq:str_linear}
    \|P_N e^{it \Delta}\phi\|_{L^p(\T\times \T^3)}
    \ls N^{\frac32-\frac5p} \|P_N \phi\|_{L^2(\T^3)}.
  \end{equation}
  More generally, for all $C \in \mathcal{C}_N$ we have
  \begin{equation}\label{eq:str_linear_cubes}
    \|P_C e^{it \Delta}\phi\|_{L^p(\T\times \T^3)}
    \ls N^{\frac32-\frac5p} \|P_C \phi\|_{L^2(\T^3)}.
  \end{equation}
\end{prop}
\begin{proof}[Proof/Reference]
  Estimate \eqref{eq:str_linear} follows \cite[formula 3.117]{B93a}.
  Let $C \in \mathcal{C}_N$ with center $\xi_0$.
  We first apply a Galilean transformation (as in \cite[formulas
  5.7-5.8]{B93a}) to shift the center of the cube $C$ to the
  origin and replace $C$ by $C_0 = C - \xi_0$. Denoting $\phi_0(x) =
  e^{-ix\cdot \xi_0} \phi(x)$ we have $\widehat \phi_0(\xi) = \widehat
  \phi(\xi+\xi_0)$ and therefore $\| P_C \phi\|_{L^2(\T^3)} = \| P_{C_0}
  \phi_0\|_{L^2(\T^3)}$.  Furthermore, we can relate the corresponding
  solutions to the Schr\"odinger equations. We have
  \[
  P_{C_0} e^{it \Delta} \phi_0(t,x) = \sum_{\xi \in C_0} e^{i(x\cdot
    \xi-t|\xi|^2)}\widehat{\phi_0}(\xi) = \sum_{\xi \in C} e^{i(x\cdot
    (\xi-\xi_0)-t|\xi-\xi_0|^2)}\widehat{\phi}(\xi),
  \]
  therefore by rewriting the phase
  \[
  x\cdot(\xi-\xi_0)-t|\xi-\xi_0|^2 = (x+2t\xi_0) \cdot \xi - t|\xi|^2 -
  x\cdot \xi_0 - t |\xi_0|^2
  \]
  we obtain
  \[
  P_{C_0} e^{it \Delta} \phi_0(t,x) = e^{-i (x \xi_0 + t |\xi_0|^2)}
  P_{C} e^{it \Delta} \phi(t,x+2t\xi_0)
  \]
  Thus $\| P_C e^{it \Delta}\phi\|_{L^p(\T \times \T^3)} = \| P_{C_0}e^{it \Delta}
  \phi_0\|_{L^p(\T \times \T^3)}$. The bound \eqref{eq:str_linear_cubes} with $C$ replaced by $C_0$ follows from \eqref{eq:str_linear}.
\end{proof}

\begin{coro}
  Let $p>4$. For all $N\geq 1$ and $C \in \mathcal{C}_N$ we have
  \begin{equation}\label{eq:str_emb_cubes}
    \|P_C u\|_{L^p(\T\times \T^3)}
    \ls N^{\frac32-\frac5p} \|P_C u\|_{U^p_{\Delta} L^2}.
  \end{equation}
\end{coro}
\begin{proof}
  Due to the atomic structure of $U^p$ it suffices to consider a
  piecewise linear solution $a$, i.e.
  \[
  a(t)=\sum_{k=1}^{K}\chi_{[t_{k-1},t_{k})}(t) e^{it\Delta} \phi_{k-1}
  \text{ with } \sum_{k=0}^{K}\|\phi_{k-1}\|_{L^2}^{p}=1
  \]
  We apply \eqref{eq:str_linear_cubes} on each interval
  $[t_{k-1},t_{k})$, and then sum up with respect to $k$.
\end{proof}

In the remainder of this section we will refine Bourgain's estimate
\eqref{eq:str_linear} in the case $p>4$ by proving stronger
bounds for solutions which are frequency-localized to $3d$ rectangles.

Let $\mathcal{R}_{M}(N)$ be the collection of all sets in $\Z^3$ which
are given as the intersection of a cube of sidelength $2N$ with strips
of width $2M$, i.e. the collection of all sets of the form
\[
(\xi_0+[-N,N]^3)\cap \{ \xi \in \Z^3: |a\cdot \xi -A|\leq M \}
\]
with some $\xi_0 \in \Z^3, a \in \R^3$, $|a|=1$, $A \in \R$.
\begin{prop}\label{prop:str_linear_impr_linfty}
  For all $1\leq M\leq N$ and $R \in \mathcal{R}_M(N)$ we have
  \begin{equation}\label{eq:str_linear_impr_linfty}
    \|P_R  e^{it \Delta}\phi\|_{L^\infty (\T\times \T^3)}
    \ls M^{\frac12}N \|P_R \phi\|_{L^2(\T^3)}.
  \end{equation}
\end{prop}
\begin{proof}
For fixed $t$ we have the uniform bound
\[
\|P_R  e^{it \Delta}\phi\|_{L^\infty (\T^3)}\leq \sum_{\xi \in \Z^3\cap R}|\widehat{\phi}(\xi)|
\leq (\#(\Z^3\cap R))^{\frac12}\|P_R \phi\|_{L^2(\T^3)}.
\]
  We cover the rectangle by $\sim N^2/M^2$ cubes of side-length $M$ and
  faces parallel to the coordinate planes. Each such cube contains $\sim M^3$ lattice points.
  Altogether, we arrive at the upper bound
  \[\#(\Z^3\cap R)\ls M N^2,\] which shows the
  claim.
\end{proof}

The estimates \eqref{eq:str_linear_cubes} and
\eqref{eq:str_linear_impr_linfty} and H\"older's inequality imply

\begin{coro}\label{coro:str_linear_impr}
  Let $p>4$ and $0<\delta<\frac12-\frac2p$. For all $1\leq M\leq N$
  and $R \in \mathcal{R}_M(N)$ we have
  \begin{equation}\label{eq:str_linear_impr}
    \|P_R e^{it \Delta}\phi\|_{L^p(\T\times \T^3)}
    \ls N^{\frac32-\frac5p} \left(\frac{M}{N}\right)^{\delta}\|P_R \phi\|_{L^2(\T^3)}.
  \end{equation}
\end{coro}

We conclude the section with our key trilinear estimate:

\begin{prop}\label{prop:l2_est_linear}
  There exists $\delta>0$ such that for any $N_1\geq N_2 \geq N_3\geq
  1$ and any interval $I\subset [0,2\pi]$ we have
  \begin{equation}\label{eq:l2_est_linear}
    \|\prod_{j=1}^3P_{N_j} u_j\|_{L^2(I\times \T^3)}\ls
    N_2N_3\max\Big\{\frac{N_3}{N_1},\frac{1}{N_2}\Big\}^\delta
    \prod_{j=1}^3\|P_{N_j}u_j\|_{Y^0}.
  \end{equation}
\end{prop}
\begin{proof}
  We begin the proof with some simplifications.  First of all, it is
  enough to consider the case $I=[0,2\pi]$. Second, given a partition $\Z^3
  = \cup C_j$ into cubes $C_j \in \mathcal C_{N_2}$, the outputs
  $P_{N_1} P_{C_j} u_1 P_{N_2} u_2 P_{N_3} u_3$ are almost orthogonal because the
  spatial Fourier-support of $u_2 P_{N_3} u_3$ is contained in at most finitely many cubes of
  sidelength $N_2$. Additionally, we have
  \[
  \| u\|_{Y^0}^2 = \sum_{j} \|P_{C_j} u\|_{Y^0}^2,
  \] which implies that it suffices to prove \eqref{eq:l2_est_linear} in the case when the
  first factor is further restricted to a cube $C \in
  \mathcal{C}_{N_2}$,
  \begin{equation}\label{eq:l2_est_lineara}
    \|P_{N_1} P_C u_1 P_{N_2} u_2 P_{N_3} u_3\|_{L^2}\ls
    N_2N_3 \left(\frac{N_3}{N_1}+\frac{1}{N_2}\right)^\delta
    \prod_{j=1}^3\|P_{N_j}u_j\|_{Y^0},
  \end{equation}
  where $L^2=L^2([0,2\pi]\times\T^3)$.

  In \eqref{eq:l2_est_lineara}, thanks to Proposition
  \ref{prop:x_in_y}, we are allowed to replace $Y^0$ by $V^2_\Delta
  L^2$. Then \eqref{eq:l2_est_lineara} follows by interpolation, via
  our interpolation Lemma~\ref{lem:interpol}, from the following two
  trilinear estimates:
  \begin{equation}\label{eq:l2_est_cubes}
    \begin{split}
      \|P_C P_{N_1} u_1 P_{N_2}u_2 P_{N_3}u_3\|_{L^2} \ls{} N_2N_3
      \left(\frac{N_3}{N_2}\right)^{\delta'}
      \prod_{j=1}^3\|P_{N_j}u_j\|_{U^{4}_\Delta L^2}
    \end{split}
  \end{equation}
  for $0<\delta'<\frac12$, respectively
  \begin{equation}\label{eq:l2_est_u2}
    \|P_C P_{N_1} u_1  P_{N_2}u_2 P_{N_3}u_3\|_{L^2}
    \ls   N_2N_3\left(\frac{N_3}{N_1}+\frac{1}{N_2}\right)^{\delta''}
    \prod_{j=1}^3\|P_{N_j}u_j\|_{U^2_\Delta L^2},
  \end{equation}
for some $\delta''>0$.

  The first bound \eqref{eq:l2_est_cubes} follows from H\"older's
  inequality with $4<p<5$ and $q$ such that $2/p+1/q=1/2$ and
  \eqref{eq:str_emb_cubes}
  \begin{align*}
    &\|P_C P_{N_1} u_1 P_{N_2}u_2 P_{N_3}u_3\|_{L^2}
    \leq\|P_C P_{N_1} u_1 \|_{L^p} \|P_{N_2}u_2\|_{L^p}\| P_{N_3}u_3\|_{L^q}\\
    \ls{}&N_2^{3-\frac{10}p}N_3^{\frac32-\frac5q}\|P_{N_1}u_1\|_{U^{p}_\Delta
      L^2}\|P_{N_2}u_2\|_{U^{p}_\Delta
      L^2}\|P_{N_3}u_3\|_{U^{q}_\Delta L^2},
  \end{align*}
  and the embedding $U^{4}_\Delta L^2 \hookrightarrow U^{p}_\Delta
  L^2\hookrightarrow U^{q}_\Delta L^2$.

  For the second bound \eqref{eq:l2_est_u2} we can use the atomic
  structure of the $U^2$ spaces (see e.g. \cite[Proposition 2.19]{HHK09} for more details on this point) to reduce the problem to the similar
  estimate for the product of three solutions to the linear
  Schr\"odinger equation $u_j=e^{it\Delta}\phi_j$:
  \begin{equation}\label{eq:l2_est_linear_b}
    \|P_C P_{N_1} u_1  P_{N_2}u_2 P_{N_3}u_3\|_{L^2}  \ls
    N_2N_3\left(\frac{N_3}{N_1}+\frac{1}{N_2}\right)^{\delta''}\prod_{j=1}^3\|P_{N_j}\phi_j\|_{L^2}.
  \end{equation}

  Let $\xi_0$ be the center of $C$.  We partition $C = \cup R_k$ into
  almost disjoint strips of width $M=\max\{N_2^2/N_1,1\}$ which are
  orthogonal to $\xi_0$,
  \[
  R_k=\Big\{\xi\in C :\xi\cdot \xi_0 \in [|\xi_0| M k, |\xi_0| M(k+1))
  \Big\}, \qquad |k| \approx N_1/M
  \]
  Each $R_k$ is the intersection of a cube of sidelegth $2N_2$ with a strip of width $M$, and we decompose
  \[
  P_C P_{N_1} u_1 P_{N_2}u_2 P_{N_3}u_3 = \sum_k P_{R_k} P_{N_1} u_1
  P_{N_2}u_2 P_{N_3}u_3
  \]
  and claim that the summands are almost orthogonal in $L^2(\T \times
  \T^3)$. This orthogonality no longer comes from the spatial
  frequencies, but from the time frequency.  Indeed, for $\xi_1 \in
  R_k$ we have
  \[
  \xi_1^2 = \frac1{|\xi_0|^{2}}|\xi_1 \cdot \xi_0|^2 + |\xi_1-\xi_0|^2
  - \frac1{|\xi_0|^{2}}|(\xi_1-\xi_0) \cdot \xi_0|^2 = M^2 k^2 + O(M^2
  k)
  \]
  since $N_2^2 \lesssim M^2 k$. The second and third factor alter the
  time frequency by at most $O(N_2^2)$. Hence the expressions $P_{R_k}
  P_{N_1} u_1 P_{N_2}u_2 P_{N_3}u_3$ are localized at time frequency
  $M^2 k^2 + O(M^2 k)$ and thus are almost orthogonal,
  \[
  \| P_C P_{N_1} u_1 P_{N_2}u_2 P_{N_3}u_3\|_{L^2}^2
  \lesssim \sum_k \|P_{R_k} P_{N_1} u_1 P_{N_2}u_2
  P_{N_3}u_3\|_{L^2}^2
  \]
  On the other hand the estimates \eqref{eq:str_linear_impr} and
  \eqref{eq:str_linear_cubes} yield
  \begin{align*}
    \|P_{R_k} P_{N_1} u_1 P_{N_2}u_2 P_{N_3}u_3\|_{L^2} \ls &
    N_2^{3-\frac{10}p} \left(\frac{M}{N_2}\right)^\eps N_3^{\frac32-\frac5q} \\
    &\|P_{R_k}
    P_{N_1}\phi_1\|_{L^2}\|P_{N_2}\phi_2\|_{L^2}\|P_{N_3}\phi_3\|_{L^2},
  \end{align*}
with $4<p<5$ and $q$ such that $2/p+1/q=1/2$, and $0<\eps<1/2-2/p$.
  Then \eqref{eq:l2_est_linear_b} follows by summing up the squares
  with respect to $k$.
\end{proof}

\section{Proof of the main results}\label{sect:pr}
Before we explain the proofs of Theorems \ref{thm:main-local} and
\ref{thm:main-global} we summarize the results of the previous section
in the following nonlinear estimate.
\begin{prop}\label{prop:multilinear_est}
  Let $s\geq 1$ be fixed. Then, for all $0<T\leq 2\pi$, and $u_k\in
  X^s([0,T))$, $k=1,\ldots,5$, the estimate
  \begin{equation}\label{eq:multilinear_est}
    \Big\|\mathcal{I}(\prod_{k=1}^5
    \widetilde{u}_k)\Big\|_{X^s([0,T))}
    \ls \sum_{j=1}^5\|u_j\|_{X^s([0,T))}\prod_{\genfrac{}{}{0pt}{}{k=1}{k\not=j}}^5\|u_k\|_{X^1([0,T))},
  \end{equation}
  holds true, where $\widetilde{u}_k$ denotes either $u_k$ or
  $\overline{u}_k$.
\end{prop}
\begin{proof} Let $I$ denote the interval $I=[0,T)$, and let $N\geq 1$ be given.
  Proposition \ref{prop:inhom_est} implies $\mathcal{I}(P_{\leq
    N}\prod_{k=1}^5 \widetilde{u}_k)\in X^s(I)$ and
  \[
  \Big\|\mathcal{I}(P_{\leq N}\prod_{k=1}^5
  \widetilde{u}_k)\Big\|_{X^s(I)}\leq \sup_{\genfrac{}{}{0pt}{}{v \in
      Y^{-s}(I):}{\|v\|_{Y^{-s}}=1}} \int_0^{2\pi} \int_{\T^3}P_{\leq
    N}\prod_{k=1}^5 \widetilde{u}_k \, \overline{v} dx dt.
  \]
  We denote $u_0 = P_{\leq N} v$.  Then we need to prove the
  multilinear estimate
  \begin{equation}\label{6lin_r}
    \left| \int_{I \times \T^3} \prod_{k=0}^5 \tilde u_k \ dx dt\right|
    \lesssim  \| u_0\|_{Y^{-s}(I)}  \sum_{j=1}^5\left(\|u_j\|_{X^s(I)}\prod_{\genfrac{}{}{0pt}{}{k=1}{k\not=j}}^5\|u_k\|_{X^1(I)}\right)
  \end{equation}
Once \eqref{6lin_r} is obtained the claimed estimate follows by letting
$N \to\infty$. We consider extensions to
$\R$ of $u_k$ which we will also denote with $u_k$ in the sequel
of the proof,
$k=0,\ldots,5$, and \eqref{6lin_r} reduces to
 \begin{equation}\label{6lin}
    \left| \int_{I \times \T^3} \prod_{k=0}^5 \tilde u_k \ dx dt\right|
    \lesssim  \| u_0\|_{Y^{-s}}  \sum_{j=1}^5\left(\|u_j\|_{X^s}\prod_{\genfrac{}{}{0pt}{}{k=1}{k\not=j}}^5\|u_k\|_{X^1}\right).
  \end{equation}
  We dyadically decompose
  \[
  \widetilde{u}_k=\sum_{N_k\geq 1} P_{N_k} \widetilde{u}_k.
  \]
  In order for the integral in \eqref{6lin} to be nontrivial, the two
  highest frequencies must be comparable. Then, by the Cauchy-Schwarz
  inequality and symmetry it suffices to show that (with
  $L^2=L^2(I\times \T^3)$)
  \begin{equation}\label{6lina}
    \begin{split}
      S&= \sum_{ \mathcal
        N}\|P_{N_1}\widetilde{u}_1P_{N_3}\widetilde{u}_3P_{N_5}\widetilde{u}_5\|_{L^2}\|P_{N_0}\widetilde{u}_0P_{N_2}\widetilde{u}_2P_{N_4}\widetilde{u}_4\|_{L^2}
      \\& \lesssim \| u_0\|_{Y^{-s}}
      \sum_{j=1}^5\|u_j\|_{X^s}\prod_{\genfrac{}{}{0pt}{}{k=1}{k\not=j}}^5\|u_k\|_{X^1},
    \end{split}
  \end{equation}
  where $\mathcal N$ is as the set of all $6$-tuples
  $(N_0,N_1,\ldots,N_5)$ of dyadic numbers $N_i\geq 1$ satisfying
  \[
  N_5\leq \ldots\leq N_1, \quad \max\{N_0,N_2\} \approx N_1.
  \]
  We subdivide the sum into two parts $S = S_1+S_2$:

  {\bf Part $\mathcal N_1$:} $N_2\leq N_0 \approx N_1$: Proposition
  \ref{prop:l2_est_linear} implies
  \[
  S_1 \lesssim \sum_{\mathcal N_1} N_2N_3N_4N_5 \left(\frac{N_5}{N_1}+
    \frac{1}{N_3}\right )^\delta \left(\frac{N_4}{N_0}+
    \frac{1}{N_2}\right )^\delta \prod_{k=0}^5
  \|P_{N_k}u_k\|_{Y^0}.
  \]
  Using Cauchy-Schwarz, we easily sum up with respect to $N_2$, $N_3$, $N_4$ and $N_5$, and obtain
  \[
  S_1 \lesssim \sum_{N_0 \approx N_1} \| P_{N_0}u_0\|_{Y^0} \|
  P_{N_1}u_1\|_{Y^0} \prod_{k=2}^5 \|u_k\|_{Y^1}.
  \]
  Another application of Cauchy-Schwarz with respect to $N_1$ yields
  \[
  S_1 \lesssim \| u_0\|_{Y^{-s}} \| u_1\|_{Y^{s}} \prod_{k=2}^5
  \|u_k\|_{Y^1},
  \]
  as needed.

  {\bf Part $\mathcal N_2$:} $N_0 \leq N_2 \approx N_1$.  Proposition
  \ref{prop:l2_est_linear} implies
  \[
  S_2 \lesssim \sum_{\mathcal N_2} N_0 N_3 N_4 N_5
  \left(\frac{N_5}{N_1}+ \frac{1}{N_3}\right )^\delta \prod_{k=0}^5
  \|P_{N_k}u_k\|_{Y^0}.
  \]
  By Cauchy-Schwarz we sum with respect to $N_3$, $N_4$ and $N_5$ and obtain
  \[
  S_2 \lesssim \sum_{N_0 \leq N_1 \approx N_2} N_0 \|
  P_{N_0}u_0\|_{Y^0} \| P_{N_1}u_1\|_{Y^0} \|
  P_{N_2}u_2\|_{Y^0} \prod_{k=3}^5 \|u_k\|_{Y^1}.
  \]
  For the remaining sum we apply directly Cauchy-Schwarz with respect
  to $N_0$ to obtain
  \[
  S_2 \lesssim \sum_{N_1\approx N_2} N_1^{s+1} \|u_0\|_{Y^{-s}} \|
  P_{N_1}u_1\|_{Y^0} \| P_{N_2}u_2\|_{Y^0} \prod_{k=3}^5
  \|u_k\|_{Y^1}.
  \]
  Finally, we apply Cauchy-Schwarz with respect to $N_1$ to obtain
  \[
  S_2 \lesssim \| u_0\|_{Y^{-s}} \| u_1\|_{Y^{s}} \prod_{k=2}^5
  \|u_k\|_{Y^1}.
  \]
  The proof of the proposition is complete.
\end{proof}

\begin{proof}[Proof of Theorem \ref{thm:main-local}] We split the
  proof into several parts. The general strategy of the proof is
  well-known, see e.g. \cite{CW90,T06}. We study the case $s=1$ in
  detail.

  {\em Part 1: Small data.}  From \eqref{eq:multilinear_est} we obtain
  \begin{equation}\label{eq:nonl_est}
    \begin{split}
      &\Big\|\mathcal{I}(|u|^4u-|v|^4v)\Big\|_{X^1([0,T))}
      \\
      \leq{} & c
      (\|u\|^4_{X^1([0,T))}+\|v\|^4_{X^1([0,T))})\|u-v\|_{X^1([0,T))},
    \end{split}
  \end{equation}
  for all $0<T\leq 2\pi$ and $u,v\in X^1([0,T))$.

  For parameters $\eps>0,\delta>0$, consider the sets
  \begin{align*}
    B_{\eps}:=&\{\phi \in H^1(\T^3): \|\phi\|_{H^1}\leq \eps\},\\
    D_{\delta}:=&\{u \in X^1([0,2\pi))\cap C([0,2\pi);H^1(\T^3)):
    \|u\|_{X^1([0,2\pi))}\leq \delta\}.
  \end{align*}
  The set $D_{\delta}$ is closed in $X^1([0,2\pi))$, hence it is a
  complete space.  With $L(\phi)=e^{it\Delta} \phi$ and
  $NL(u)=-i\mathcal{I}(|u|^4u)$ we aim to solve the equation
  \[
  u=L(\phi)+NL(u)
  \]
  by the contraction mapping principle in $D_{\delta}$, for $\phi\in
  B_{\eps}$. We have
  \[
  \|L(\phi)+NL(u)\|_{X^1([0,2\pi))}\leq \eps +c\delta^5\leq \delta,
  \]
  by choosing
  \begin{equation}\label{eq:choice_eps}
    \delta=(4c)^{-\frac14} \text{ and } \eps=\delta/2.
  \end{equation}
  With the same choice it follows that
  \[
  \|NL(u)-NL(v)\|_{X^1([0,2\pi))}\leq \frac12 \|u-v\|_{X^1([0,2\pi))},
  \]
  which shows that the nonlinear map $L(\phi)+NL(u)$ has a unique fixed
  point in $D_{\delta}$ (Concerning uniqueness in the full space, see
  below). Similarly, for $\phi,\psi \in B_{\eps}$ and the
  corresponding fixed points $u,v \in D_{\delta}$ it follows
  \[
  \|u-v\|_{X^1([0,T))}\leq
  \|\phi-\psi\|_{H^1}+2c\delta^4\|u-v\|_{X^1([0,T))},
  \]
  which yields the Lipschitz dependence.

  {\em Part 2: Large data.}  Let $r>0$ and $N\geq 1$ be given. For
  parameters $\eps,\delta,R,T$, with $0<\eps\leq r$, $0<\delta\leq R$,
  consider the sets
  \begin{align*}
    B_{\eps,r}:=&\{\phi \in H^1(\T^3): \|\phi_{>N}\|_{H^1}\leq \eps, \, \|\phi\|_{H^1}\leq r\},\\
    D_{\delta,R,T}:=&\{u \in X^1([0,T))\cap C([0,T);H^1(\T^3)):\\
    &\qquad \qquad \|u_{>N}\|_{X^1([0,T))}\leq \delta, \,
    \|u\|_{X^1([0,T))}\leq R\},
  \end{align*}
  where $f_{>N}=(I-P_{\leq N})f$.  For $\phi \in B_{\eps,r}$ we have
  \begin{equation}\label{eq:self1}
    \|[L(\phi)+NL(u)]_{>N}\|_{X^1([0,T))}\leq \eps+\|[NL(u)]_{>N}\|_{X^1([0,T))}.
  \end{equation}
  We split $NL(u)=NL_1(u_{\leq N},u_{>N})+NL_2(u_{\leq N},u_{>N})$, such
  that $NL_1$ is at least quadratic in $u_{>N}$ and $NL_2$ is at least
  quartic in $u_{\leq N}$. In view of \eqref{eq:multilinear_est} we obtain
  \begin{equation}\label{eq:self2}
    \|NL_1(u_{\leq  N},u_{>N})\|_{X^1([0,T))}\leq c \delta^2 R^3.
  \end{equation}
  Concerning the estimate for $NL_2$ we recall (see Remark \ref{rem:l1bound}) that it suffices to control the nonlinearity in $L^1([0,T);H^1)$. The latter follows from a Sobolev embedding type argument
  \begin{equation}\label{eq:self3}
    \begin{split}
     & \|NL_2(u_{\leq  N},u_{>N})\|_{X^1([0,T))}
\leq   c_1\|u\|_{L^\infty([0,T); H^1)}\|u_{\leq
        N}\|^4_{L^4([0,T);L^\infty)}\\
&+c_1 N\|u\|_{L^\infty([0,T); L^6)}\|u_{\leq N}\|^4_{L^4([0,T);L^{12})}
\leq   c_2 N^{2}T R^5.
    \end{split}
  \end{equation}
  Similarly to \eqref{eq:self2}, one can show that
  \[
  \|NL_1(u_{\leq N},u_{>N})-NL_1(v_{\leq N},v_{>N})\|_{X^1([0,T))}\leq
  c_3 \delta R^3 \|u-v\|_{X^1([0,T))}.
  \]
  Moreover, as in \eqref{eq:self3}, it follows
  \[
  \|NL_2(u_{\leq N},u_{>N})-NL_2(v_{\leq N},v_{>N})\|_{X^1([0,T))}\leq
  c_4 N^{2}R^4 T \|u-v\|_{X^1([0,T))}.
  \]
  Let $C\geq 1$ be the maximum of $c,c_1,\ldots,c_4$. By choosing the
  parameters
  \begin{equation}\label{eq:para}
    R:=4r, \; \delta:=(8CR^3)^{-1}, \; \eps:=\delta/2, \; T:=\delta(8CR^5N^{2})^{-1},
  \end{equation}
  we have found that for any $\phi \in B_{\eps,r}$ the map
  \[
  L(\phi)+NL(u):D_{\delta,R,T}\to D_{\delta,R,T}
  \]
  is a strict contraction, which has a unique fixed point $u$ and
  $\phi\mapsto u$ is Lipschitz continuous with constant $2$.

  {\em Part 3: Conclusion.} Concerning uniqueness, we note that by
  translation invariance in $t$ it suffices to consider
  \[u,v \in X^1([0,T))\cap C([0,T);H^1(\T^3)) \text{ such that }
  u(0)=v(0),
  \]
  and show $u=v$ for arbitrarily small $T>0$. Indeed, this follows
  from the uniqueness of the fixed point in balls of arbitrary radius
  already proved in Part 2.

  Now, let $\phi_\ast \in H^1(\T^3)$ and $\eps>0$ be as in
  \eqref{eq:para} for given $r=2\|\phi_\ast\|_{H^1}$ (notice that this
  choice is independent of $N$). Choose $N\geq 1$ large enough such
  that $\|[\phi_\ast]_{>N}\|_{H^1}\leq \eps/2$. Then, for all $\phi
  \in B_{\eps/2}(\phi_\ast)$ we have $\phi \in B_{\eps,r}$, and by
  Part 2 we find $T=T(r,N)>0$ and a unique solution $u \in
  X^1([0,T))\cap C([0,T);H^1(\T^3))$ which depends Lipschitz
  continuously on the initial data $\phi$.

  Finally, we remark that by the tame estimate \eqref{eq:multilinear_est} we can easily
  prove persistence of higher order Sobolev regularity, with a
  (conceivably shorter) lifespan depending on
  $\|\phi\|_{H^s}$. However, an a posteriori iteration argument shows
  that the time of existence must be at least $T(\phi_\ast)$.
\end{proof}

\begin{proof}[Proof of Theorem \ref{thm:main-global}]
  By part 1 and 3 of the previous proof it suffices to provide a
  suitable a priori bound on the solution in $H^1$.

  In the \emph{defocusing case}, this obviously follows from
  conservation laws \eqref{eq:energy-cons} and \eqref{eq:l2-cons} the
  Sobolev embedding $H^1(\T^3)\hookrightarrow L^6(\T^3)$:
  \[
  \|u(t)\|^2_{L^2}+\|\nabla u \|_{L^2}^2\leq 2 E(u(0))+2 M(u(0)) \leq
  \|u(0)\|_{H^1}^2+d^2\|u(0)\|_{H^1}^6.
  \]
  If $\|u(0)\|_{H^1}$ is chosen to be small enough, it follows that
  for $\eps$ as in \eqref{eq:choice_eps} the solution satisfies
  $\|u(t)\|_{H^1}\leq \eps$ on any interval of existence, and we can
  iterate the argument from Part 1 indefinitely.

  In the \emph{focusing case}, we combine the conservation laws
  \eqref{eq:energy-cons} and \eqref{eq:l2-cons} with a continuity
  argument, based on the inequality
  \begin{align*}
    \|u(t)\|^2_{H^1}=& 2 E(u(0))+2 M(u(0))+\frac13\|u(t)\|_{L^6}^6\\
    \leq &\|u(0)\|_{H^1}^2+d^2\|u(0)\|_{H^1}^6+d^2\| u(t)\|^6_{H^1},
  \end{align*}
  which is valid for any solution $u \in X^1([0,2\pi))$.

  We consider $f(x)=x-d^2x^3$ for $x\geq 0$, which increases from $0$
  to its maximum value $2(3\sqrt{3}d)^{-1}$ and satisfies $f(x)\geq
  \frac23 x$ on the interval $I=[0,(\sqrt{3} d)^{-1}]$. We have shown
  above that $f(\|u(t)\|^2_{H^1})< \eps_0^2$, for all $t \in [0,2\pi)$
  and all initial data satisfying \[\|u(0)\|_{H^1}^2+d^2\|u(0)\|_{H^1}^6< \eps_0^2.\]
  By choosing $\eps_0^2=\min\{2(3\sqrt{3}d)^{-1},2/3\eps^2\}$, with $\eps$ as in
  \eqref{eq:choice_eps}, the continuity of $t\mapsto \|u(t)\|^2_{H^1}$
  implies that $\|u(t)\|^2_{H^1}\in I$ for all $t\in [0,2\pi)$, and
  therefore $\|u(t)\|^2_{H^1}\leq 3/2\eps_0^2 \leq \eps^2$ for all $t
  \in [0,2\pi)$, which allows us to iterate the small data local
  well-posedness argument.
\end{proof}

\bibliographystyle{amsplain} \label{sect:refs}\bibliography{nls-refs}

\end{document}